\documentclass[10pt]{article}
 \usepackage[margin=1in]{geometry} 
\usepackage{amsmath,amsthm,amssymb,amsfonts}
 \usepackage{amssymb}

\numberwithin{equation}{section}

\usepackage[utf8]{inputenc}
\usepackage[pagewise]{lineno}

\usepackage{hyperref}
 \usepackage{csquotes}
\usepackage[utf8]{inputenc}
\usepackage[english]{babel}
\usepackage{xcolor}

\providecommand{\keywords}[1]
{
  \small	
  \textbf{\textit{Keywords:}} #1
}

\newcommand{\MSC}[1]{%
  \small
  \textbf{\textit{Mathematics Subject Classification:}} #1
}
\title{Global existence of classical solutions to a system modeling propagation of urban crime with quadratic logistic damping}
\author{
    Minh Le\thanks{ Westlake Institute for Advanced Study, Westlake University, 600 Dunyu
Road, 310030 Hangzhou, Zhejiang, China
 \texttt{(leminh@westlake.edu.cn)}} 
}
\date{}

\begin{document}
\maketitle

\begin{abstract}
We are concerned with the following partial differential equations arising 
from urban crime modeling:
\begin{equation} \label{main}
    \begin{cases}
        u_t = \Delta u - \chi \nabla \cdot \left( u \dfrac{\nabla v}{v} \right) - uv + B_1 + ru - \mu u^2, \\[4pt]
        v_t = \Delta v - v + uv + B_2,
    \end{cases}
\end{equation}
under no-flux boundary conditions in a smoothly bounded domain 
$\Omega \subset \mathbb{R}^n$ with $n \geq 2$, where $\chi$, $r$, and $\mu$ 
are positive constants. It is shown in this paper that if the nonnegative 
source terms $B_1$ and $B_2$ are sufficiently regular and
\begin{equation*}
    \mu > \frac{3}{n} + \frac{1}{n}\left( \chi n 
    + \frac{n(\chi(n-1)-2)^2}{4(n-1)} \right)^{\frac{n+1}{n}} 
    \cdot \left( \frac{(2n+\sqrt{n})^2}{2n-1} \right)^{\frac{1}{n}} 
    + n^{12n+3},
\end{equation*}
then the system \eqref{main} with suitably regular initial data possesses a global classical solution. Moreover, under the assumption that
\begin{equation*} \label{cond.B2}
    \inf_{t>0} \int_\Omega B_2(\cdot,t) > 0,
\end{equation*}
the solutions are uniformly bounded in time.
\end{abstract}

\\
\MSC{35K55, 35B40, 35Q91}

\numberwithin{equation}{section}
\newtheorem{theorem}{Theorem}[section]
\newtheorem{lemma}[theorem]{Lemma}
\newtheorem{remark}{Remark}[section]
\newtheorem{Prop}{Proposition}[section]
\newtheorem{Def}{Definition}[section]
\newtheorem{Corollary}{Corollary}[theorem]
\allowdisplaybreaks

\section{Introduction}

In this paper, we consider the following system of partial differential equations in a smoothly bounded domain $\Omega \subset \mathbb{R}^n$ where $n \geq 2$ :
 \begin{equation} \label{1}
      \begin{cases}
         u_t = \Delta u - \chi \nabla \cdot ( u \frac{\nabla v}{v})-uv+B_1(x,t)+ru -\mu u^2 ,  \qquad &\text{in } \Omega \times (0,\infty), \\
          v_t=  \Delta v -v+ uv+B_2(x,t), \qquad &\text{in } \Omega \times (0,\infty),  
      \end{cases}
\end{equation}
where $\chi,r,$ and $\mu$ are positive parameters. The system satisfies the homogeneous Neumann boundary condition
\begin{equation} \label{bdry}
    \frac{\partial u}{\partial \nu} =\frac{\partial v}{\partial \nu} =0 ,\qquad x \in \partial \Omega, \, t>0,
\end{equation}
and initial conditions
\begin{equation} \label{initial}
    \begin{cases}
    \displaystyle u_0 \in C^{0} (\bar{\Omega}) \text{ is nonnegative with $\int_\Omega u_0 >0$,}  \\
    \displaystyle v_0 \in W^{1, \infty}(\Omega) \;\; \text{and $v_0> 0$ in $\bar{\Omega}.$} 
    \end{cases}
\end{equation}

In the particular case where $n=2$, $\chi=2$, and $r=\mu=0$, the system \eqref{1} was originally introduced in \cite{Short2008} and \cite{Short2010} as a mathematical model describing the spatio-temporal evolution of urban crime. Within this framework, the function $u=u(x,t)$ denotes the density of criminal agents at position $x$ and time $t$. The key behavioral assumption is that these agents tend to move toward regions where an abstract quantity $v=v(x,t)$, referred to as the attractiveness value, is higher. In other words, criminal activity is modeled as being attracted to areas that are perceived as more attractive or rewarding. The fundamental hypotheses underlying the model developed in \cite{Short2008} are grounded in well-established and statistically supported behavioral patterns. These patterns are deeply connected to the so-called broken windows theory, as discussed in \cite{Kelling1982}, which posits that visible signs of disorder and neglect in an environment can encourage further criminal activity. Moreover, the model also captures phenomena that are commonly observed in criminology under the names of repeat and near-repeat victimization, as documented in \cite{Johnson1997}. These effects describe the empirical tendency for crime to recur at or near locations that have recently been victimized, reflecting a clustering behavior that the model is designed to reproduce.

Let us briefly review some mathematical results concerning the global existence and boundedness of solutions to the system \eqref{1}. In the absence of a logistic source, the existence and uniqueness of global solutions for arbitrary $\chi>0$ were established in \cite{Winkler+Rodriguez} when $n=1$. In higher dimensions $n \geq 2$, it was demonstrated in \cite{Marcel2017} that solutions exist globally in time under the restriction that $\chi< \frac{2}{n}$. Moreover, when $B_1$ and $B_2$ are constants, the global boundedness and asymptotic behavior were established in \cite{Shen+Li} under the same condition on $\chi$. Under a smallness assumption on the initial data and on $B_1$ and $B_2$, the global existence of classical solutions was demonstrated in \cite{MR4257576} in dimension two and in \cite{MR4241613} for higher dimensions. Under the radial symmetry assumption on the initial data, global renormalized solutions were constructed in \cite{MR4002172} in dimension two for any $\chi>0$, and in \cite{MR4397171} in dimension three when $\chi \in (0, \sqrt{3})$. The existence of global generalized solutions in dimension two was proven in \cite{MR4816440}.

In the presence of a logistic source, $ru- \mu u^\alpha$, it was shown in \cite{MR4094537} that the system admits a global generalized solution in dimension two when $\alpha=2$; moreover, global classical solutions exist when $\alpha>2$. This result was later improved in \cite{MR4420141}, where the authors demonstrated that global classical solutions exist when $\alpha>2$ and $n \leq 4$, and when $\alpha>1+\frac{n}{4}$ for $n \geq 5$. At this point, a very fundamental and reasonable question arises: can the quadratic logistic damping, that is, the case $\alpha=2$, ensure the global solvability of classical solutions to the system \eqref{1}? In this paper, we investigate this question and show that the answer is yes. Before stating our main theorem, let us clarify the following assumptions and definitions.

Throughout the sequel we shall assume that the given source terms $B_1$ and $B_2$ for criminal agents and attractiveness are suitably regular in the sense that 
\begin{equation} \label{B}
    \begin{cases}
        B_1 \in C^1\left ( \bar{\Omega}\times[0, \infty) \right )  \text{is nonnegative and bounded, and that }\\
        B_2\in C^2\left ( \bar{\Omega}\times[0, \infty) \right ) \cap L^\infty \left ((0, \infty); W^{1, \infty}(\Omega) \right ) \text{is nonnegative. }
    \end{cases}
\end{equation}
 We may sometimes assume further that 
\begin{equation} \label{B2}
    \inf_{t>0}\int_\Omega B_2(\cdot,t ) >0. 
\end{equation}
We define
   \begin{align}\label{mu}
      \mu_0(p)= \frac{3}{p}+ \frac{1}{p}\left ( \chi p + \frac{p(\chi(p-1)-2)^2}{4(p-1)} \right )^{\frac{p+1}{p}} \cdot \left ( \frac{(2p+\sqrt{n})^2}{2p-1} \right )^\frac{1}{p}+n^{8p+2}p^{4p+1}.
  \end{align}
  and 
  \begin{align} \label{mu*}
      \mu_*&:= \mu_0 \left ( n \right ) =  \frac{3}{n}+ \frac{1}{n}\left ( \chi n + \frac{n(\chi(n-1)-2)^2}{4(n-1)} \right )^{\frac{n+1}{n}} \cdot \left ( \frac{(2n+\sqrt{n})^2}{2n-1} \right )^\frac{1}{n}+n^{12n+3}
  \end{align}

Let us now state our main result.

  \begin{theorem} \label{thm1}
       Let $\chi>0$, $r>0$ and $\Omega \subset \mathbb{R}^n$ with $n \geq 2$ be a smoothly bounded domain, and suppose that \eqref{initial} and \eqref{B} hold. If $\mu> \mu_*$ where $\mu_*$ is defined in \eqref{mu*}, then there exists a pair $(u,v)$ of functions
    \begin{equation*}
        \begin{cases}
            u \in C^0 \left ( \bar{\Omega}\times [0, \infty) \right ) \cap C^{2,1}\left ( \bar{\Omega}\times (0, \infty) \right ),\\
            v \in C^0 \left ( \bar{\Omega}\times [0, \infty) \right ) \cap C^{2,1}\left ( \bar{\Omega}\times (0, \infty) \right ),
        \end{cases}
    \end{equation*}
    which solve \eqref{1} classically in $\Bar{\Omega}\times [0,\infty)$, and which are such that $u>0$ and $v>0$ in $\bar{\Omega}\times [0,\infty)$. Moreover, if the condition \eqref{B2} holds then the solution is bounded in the sense that 
    \begin{align} 
        \sup_{t >0} \left \{ \left \| u(\cdot,t) \right \|_{L^\infty(\Omega)}+  \left \| v(\cdot,t) \right \|_{W^{1,\infty}(\Omega)} \right \}< \infty.
    \end{align}
  \end{theorem}

\begin{remark}
   Our result substantially improves the findings of \cite{MR4094537} and \cite{MR4420141}. In \cite{MR4094537}, global solvability of generalized solution was established only in dimension two and only for the quadratic case $\alpha=2$. In \cite{MR4420141}, the authors required the stronger damping condition $\alpha>\min\left\{2,\,1+\frac{n}{4}\right\}$. In contrast, the present work establishes global classical solvability for the quadratic case $\alpha=2$ in all dimensions $n\ge 2$, thereby covering the full range of dimensions and removing the restriction $\alpha>1+\frac{n}{4}$ entirely.
\end{remark} 

\begin{remark}
Theorem \ref{thm1} not only asserts the global existence and boundedness of solutions to the system \eqref{1} when $\mu$ is sufficiently large, but also provides a quantified estimate for how large $\mu$ must be. However, we leave open the question of determining the smallest value $\mu_*^{opt}= \mu_*^{opt}(n, \chi)>0$ such that if $\mu > \mu_*^{opt}$, then solutions are global.  
\end{remark}

The main challenges in proving Theorem \ref{thm1} are twofold. The first difficulty arises from the singular sensitivity of the chemo-attractant term, while the second stems from the nonlinear signal production term. These obstacles prevent the use of traditional approaches; for instance, if we consider the energy functional $\int_\Omega u^p$, it is impossible to handle the singularity of $v$ near zero that originates from the chemo-attractant term. Our approach offers a simple way to overcome these obstacles by considering the energy functional
\begin{align} \label{F}
    F(t):= \int_\Omega u^p(\cdot,t)v(\cdot,t)+ \int_\Omega \frac{|\nabla v(\cdot,t)|^{2p}}{v^{2p-1}(\cdot,t)} \qquad \text{for all }t>0. 
\end{align}
The second term in the functional $F$ was first introduced in \cite{Winkler2022} in the context of chemotaxis-consumption with singular sensitivity, and has since been used in various chemotaxis problems, such as \cite{Liu2025, Black2025, Li2022,TSD2026}. The first term, on the other hand, is inspired by a method developed in \cite{Wang2019}, where the author deals with a chemotaxis system with singular sensitivity and logistic source in dimension two.

The remainder of the paper is organized as follows. In Section \ref{S2}, we recall the local existence theory for system \eqref{1} and introduce several inequalities that will be used in later sections. Section \ref{S3} is devoted to establishing several key estimates of solutions, which are then utilized in Section \ref{S4} to prove the global existence and boundedness of solutions.
\section{Preliminaries} \label{S2}
In this section, we shall establish the local existence of solutions to 
system \eqref{1} and recall several useful inequalities, which will be 
applied in the sequel sections. Let us begin with the local 
well-posedness result in the following lemma.
\begin{lemma} \label{local}
    Let $\chi,r,\mu$ be positive constants and $\Omega \subset \mathbb{R}^n$ with $n \geq 2$ be a smoothly bounded domain, and suppose that \eqref{initial} and \eqref{B} hold. Then there exist $T_{\rm max} \in  (0,\infty]$ and a pair $(u,v)$ of functions
    \begin{equation*}
        \begin{cases}
            u \in C^0 \left ( \bar{\Omega}\times [0, T_{\rm max}) \right ) \cap C^{2,1}\left ( \bar{\Omega}\times (0, T_{\rm max}) \right ),\\
            v \in \bigcap_{p>n} C^0\left ( [0,T_{\rm max});W^{1,p}(\Omega) \right ) \cap C^{2,1}\left ( \bar{\Omega}\times (0, T_{\rm max}) \right ),
        \end{cases}
    \end{equation*}
    which solve \eqref{1} classically in $\Bar{\Omega}\times [0,T_{\rm max})$, and which are such that $u>0$ and $v>0$ in $\bar{\Omega}\times [0,T_{\rm max})$, and that 
    \begin{align} \label{local-1}
        \text{if }T_{\rm max}< \infty \qquad \text{then }
        \limsup_{t \to T_{\rm max}} \left \{ \left \| u(\cdot,t) \right \|_{L^\infty(\Omega)}+  \left \| \frac{1}{v(\cdot,t)} \right \|_{L^\infty(\Omega)} +  \left \| v(\cdot,t) \right \|_{W^{1,p}(\Omega)} \right \}= \infty \quad \text{for all }p>n.
    \end{align}
\end{lemma}
\begin{proof}
    This can be verified in a straightforward fashion using a standard method from the theory of cross-diffusive systems, notably those of chemotaxis type (see~\cite{Amann1989} and also~\cite{DM}).
\end{proof}
Henceforth, $(u,v)$ will always denote a solution of system \eqref{1} on 
$\Omega \times (0,T_{\rm max})$, whose existence is guaranteed by Lemma 
\ref{local}. We proceed by recalling a standard $L^p$ regularity estimate 
for parabolic equations, the proof of which can be found in 
\cite{DM}[Lemma 4.1].
\begin{lemma} \label{C52.Para-Reg}
Let $\Omega \subset \mathbb{R}^n$ with $n \geq 2$ be a bounded domain with smooth boundary. Assume that $p \geq 1$ and $q \geq 1$ satisfy
\begin{equation*}
    \begin{cases}
     q < \frac{np}{n-p}, & \text{when } p<n, \\[4pt]
     q < \infty, & \text{when } p=n, \\[4pt]
     q = \infty, & \text{when } p>n.
     \end{cases}
\end{equation*}
Suppose that $V_0 \in W^{1,q}(\Omega)$ and that $V$ is a classical solution to the following system:
\begin{equation}\label{C52.parabolic-equation}
    \begin{cases}
     V_t = \Delta V - a V + f, & \text{in } \Omega \times (0,T), \\[4pt]
     \dfrac{\partial V}{\partial \nu} = 0, & \text{on } \partial \Omega \times (0,T), \\[4pt]
     V(\cdot,0) = V_0, & \text{in } \Omega,
    \end{cases}
\end{equation}
where $a>0$ and $T\in (0,\infty]$. If $f \in L^\infty \bigl( (0,T); L^p(\Omega) \bigr)$, then $V \in L^\infty \bigl( (0,T); W^{1,q}(\Omega) \bigr)$.
\end{lemma}

Next, we introduce the following inequalities which allow us to control the singularity of $v$ near $0$.

\begin{lemma} \label{Lw2}
For any $\varphi \in C^2(\bar{ \Omega  })$ such that $\varphi>0$ in $\bar{\Omega}$ and $\frac{\partial \varphi}{\partial \nu}=0$ on $\partial \Omega$, the following inequalities hold for any $q \geq 2,$
\begin{align}
\label{Lw3-0}
    \int_\Omega \frac{|\nabla \varphi|^{q+2}}{\varphi^{q+1}}  \leq (q+ \sqrt{n})^2\int_\Omega \frac{|\nabla \varphi|^{q-2}}{\varphi^{q-3}} |D^2 \ln \varphi|^2,
\end{align}
and
\begin{align}
\label{Lw3-1}
    \int_\Omega \frac{|\nabla \varphi|^{q-2}}{\varphi^{q-1}} |D^2 \varphi|^2 \leq (q+\sqrt{n}+1)^2\int_\Omega \frac{|\nabla \varphi|^{q-2}}{\varphi^{q-3}} |D^2 \ln \varphi|^2,
\end{align}
as well as
\begin{align}
\label{Lw3}
    \int_{ \partial \Omega} \frac{|\nabla \varphi|^{q-2}}{\varphi^{q-1}}   \frac{\partial |\nabla \varphi|^2}{\partial \nu } \leq \epsilon \int_\Omega \frac{|\nabla \varphi|^{q-2}}{\varphi^{q-3}} |D^2 \ln \varphi|^2  + C\int_\Omega \varphi,
\end{align}
where $\epsilon>0$ is an arbitrary number, and $C=C(\epsilon)$ is a positive constant.
\end{lemma}

\begin{proof}
    The proof of \eqref{Lw3-0} and \eqref{Lw3-1} is provided in Lemma 3.4 in \cite{Winkler2022}. By applying Lemma 3.5 in \cite{Winkler2022} with $\eta = \frac{\epsilon}{1+(q+\sqrt{n})^2}$, we can find $C=C(\epsilon,q)>0$ such that 
    \begin{align}
\label{Lw2.1}
    \int_{ \partial \Omega} \frac{|\nabla \varphi|^{q-2}}{\varphi^{q-1}}   \frac{\partial |\nabla \varphi|^2}{\partial \nu } \leq \eta  \int_\Omega \frac{|\nabla \varphi|^{q-2}}{\varphi^{q-3}} |D^2 \ln \varphi|^2 +\eta  \int_\Omega \frac{|\nabla \varphi|^{q+2}}{\varphi^{q+1}}  + C\int_\Omega \varphi.
\end{align}
From \eqref{Lw3-0}, we obtain 
\begin{align*}
    \eta  \int_\Omega \frac{|\nabla \varphi|^{q+2}}{\varphi^{q+1}} \leq \eta (q+\sqrt{n})^2 \int_\Omega \frac{|\nabla \varphi|^{q-2}}{\varphi^{q-3}} |D^2 \ln \varphi|^2.
\end{align*}
Substituting this into \eqref{Lw2.1} yields
\begin{align*}
    \int_{ \partial \Omega} \frac{|\nabla \varphi|^{q-2}}{\varphi^{q-1}}   \frac{\partial |\nabla \varphi|^2}{\partial \nu } &\leq  \eta \left (1+ (q+\sqrt{n})^2 \right )^2  \int_\Omega \frac{|\nabla \varphi|^{q-2}}{\varphi^{q-3}} |D^2 \ln \varphi|^2  + C\int_\Omega \varphi \notag \\
    &\leq \epsilon\int_\Omega \frac{|\nabla \varphi|^{q-2}}{\varphi^{q-3}} |D^2 \ln \varphi|^2  + C\int_\Omega \varphi,
\end{align*}
which finishes the proof.
\end{proof}

\section{A priori estimates} \label{S3}

Let us commence this section by establishing the following lower bound for $v$, 
where we emphasize that the convexity assumption of $\Omega$ is not needed.
\begin{lemma} \label{low}
    For any $T\in (0,T_{\rm max})$, we have 
    \begin{align}\label{low-1}
        v(x,t) \geq e^{-T} \cdot \inf_{y\in \Omega} v_0(y) \qquad \text{for all }x\in \Omega \quad \text{and }t\in (0,T).
    \end{align}
    Moreover, if the condition \eqref{B2} holds then there exists $\delta>0$ depending only on $\inf_{x\in \Omega }v_0(x), \inf_{t>0}\int_\Omega B_2(\cdot,t)$ and $\Omega$ such that 
    \begin{align} \label{low-2}
         v(x,t) \geq \delta  \qquad \text{for all }x\in \Omega \quad \text{and }t\in (0,T_{\rm max}).
    \end{align}
\end{lemma}
\begin{proof}
   The proof of \eqref{low-1} follows immediately by applying the maximum 
principle to the second equation of \eqref{1}, while, by slightly modifying 
the argument of Lemma 3.1 in \cite{FS2018}, we obtain \eqref{low-2}.
\end{proof}

A second fundamental feature of \eqref{1} is the uniform-in-time 
boundedness of the mass functionals $\int_\Omega u$ and $\int_\Omega v$.

\begin{lemma}\label{L1est}
    There exists $C>0$ such that 
    \begin{align}\label{L1est-1}
        \int_\Omega u(\cdot,t) \leq C \qquad \text{for all }t\in (0,T_{\rm max}),
    \end{align}
    and 
     \begin{align} \label{L1est-2}
        \int_\Omega v(\cdot,t) \leq C \qquad \text{for all }t\in (0,T_{\rm max}).
    \end{align}
\end{lemma}
\begin{proof}
    Integrating the first equation and the second equation of \eqref{1} over $\Omega$, adding them together and applying Young's inequality, we obtain 
    \begin{align*}
        \frac{d}{dt} \left \{ \int_\Omega u(\cdot,t)+ \int_\Omega v(\cdot,t) \right \} +  \left \{ \int_\Omega u(\cdot,t)+ \int_\Omega v(\cdot,t) \right \} &= (r+1) \int_\Omega u - \mu \int_\Omega u^2 + \int_\Omega \left \{B_1+B_2 \right \} \notag \\
        &\leq C \qquad \text{for all } t\in (0,T_{\rm max}),
    \end{align*}
    where $C=\frac{(r+1)^2|\Omega|}{4\mu} + \left ( \left \|B_1 \right \|_{L^\infty(\Omega\times (0,\infty))}+ \left \|B_2 \right \|_{L^\infty(\Omega\times (0,\infty))}  \right )|\Omega|$. This, together with Gronwall's inequality, implies \eqref{L1est-1} and \eqref{L1est-2}, which completes the proof.
\end{proof}

We need the following auxiliary lemma to see how the second term of the functional $F$ in \eqref{F} comes into play.
\begin{lemma}\label{LW}
Let $q\geq 2$. Then
     \begin{align}\label{LW-1}
         \frac{d}{dt}\int_\Omega \frac{|\nabla v|^q}{v^{q-1}}&+\int_\Omega \frac{|\nabla v|^q}{v^{q-1}}+q \int_\Omega v^{-q+3}|\nabla v|^{q-2}|D^2 \ln v|^2 \notag \\
         &\leq q(q-2+\sqrt{n})\int_\Omega uv^{-q+2}|\nabla v|^{q-2}|D^2 v| +(q-1)^2\int_\Omega u v^{-q+1}|\nabla v|^q \notag \\
         &\quad+ \frac{q}{2}\int_{\partial \Omega} v^{-q+1}|\nabla v|^{q-2} \cdot \frac{ \partial |\nabla v|^2}{\partial \nu} \notag \\
         &\quad+q\int_\Omega v^{-q+1}|\nabla v|^{q-2} \nabla v \cdot \nabla B_2 ,
    \end{align}
    for all $t\in (0,T_{\rm max})$.
\end{lemma}

\begin{proof}
   According to the second equation in \eqref{1}, together with the identities
\(
2\nabla v \cdot \nabla \Delta v = \Delta |\nabla v|^2 - 2|D^2v|^2
\)
and
\(
\nabla |\nabla v|^2 = 2D^2 v \cdot \nabla v,
\)
several integrations by parts yield
\begin{align} \label{LW.1}
    \frac{d}{dt}\int_\Omega \frac{|\nabla v|^q}{v^{q-1}} &= q\int_\Omega v^{-q+1}|\nabla v|^{q-2} \nabla v \cdot \left \{ \nabla \Delta v- \nabla v+ \nabla (uv)+\nabla B_2  \right \} \notag \\
    &\quad-(q-1)\int_\Omega v^{-q}|\nabla v|^q \cdot \left \{ \Delta v-v+uv+B_2 \right \} \notag \\
    &= \frac{q}{2}\int_\Omega v^{-q+1}|\nabla v|^{q-2}\Delta |\nabla v|^2-q\int_\Omega v^{-q+1}|\nabla v|^q +q\int_\Omega v^{-q+1}|\nabla v|^{q-2} \nabla v \cdot \nabla (uv) \notag \\
   &\quad+q\int_\Omega v^{-q+1}|\nabla v|^{q-2} \nabla v \cdot \nabla B_2 -(q-1)\int_\Omega v^{-q}|\nabla v|^q \Delta v+(q-1)\int_\Omega v^{-q+1}|\nabla v|^q \notag \\
   &\quad-(q-1)\int_\Omega v^{-q+1}|\nabla v|^q u -(q-1)\int_\Omega v^{-q}|\nabla v|^qB_2 \notag \\
   &=-\frac{q(q-2)}{4}\int_\Omega v^{-q+1}|\nabla v|^{q-4} \left | \nabla |\nabla v|^2 \right |^2+q(q-1)\int_\Omega v^{-q}|\nabla v|^{q-2} \nabla v \cdot \nabla |\nabla v|^2 \notag \\
   &\quad-q \int_\Omega v^{-q+1}|\nabla v|^{q-2}|D^2v|^2 -q(q-1)\int_\Omega v^{-q-1}|\nabla v|^{q+2} - \int_\Omega v^{-q+1}|\nabla v|^q \notag \\
   &\quad+ \frac{q}{2}\int_{\partial \Omega} v^{-q+1}|\nabla v|^{q-2} \cdot \frac{ \partial |\nabla v|^2}{\partial \nu}+q\int_\Omega v^{-q+1}|\nabla v|^{q-2} \nabla v \cdot \nabla B_2 -(q-1)\int_\Omega v^{-q}|\nabla v|^qB_2 \notag \\
   &\quad-q(q-2)\int_\Omega uv^{-q+2}|\nabla v|^{q-4} \nabla v \cdot (D^2 v \cdot \nabla v) - q\int_\Omega u v^{-q+2}|\nabla v|^{q-2}\Delta v \notag \\
   &\quad+(q-1)^2\int_\Omega u v^{-q+1}|\nabla v|^q \qquad \text{for all }t\in (0,T_{\rm max}). 
\end{align}
    Thanks to the fact that 
    \begin{align*}
        |D^2v|^2= v^2 |D^2\ln v|^2+ \frac{1}{v} \nabla v \cdot \nabla |\nabla v|^2- \frac{|\nabla v|^4}{v^2},
    \end{align*}
    we can write 
    \begin{align} \label{LW.2}
        &-\frac{q(q-2)}{4}\int_\Omega v^{-q+1}|\nabla v|^{q-4} \left | \nabla |\nabla v|^2 \right |^2+q(q-1)\int_\Omega v^{-q}|\nabla v|^{q-2} \nabla v \cdot \nabla |\nabla v|^2 \notag \\
   &-q \int_\Omega v^{-q+1}|\nabla v|^{q-2}|D^2v|^2 -q(q-1)\int_\Omega v^{-q-1}|\nabla v|^{q+2} \notag \\
   &\quad =-q \int_\Omega v^{-q+3}|\nabla v|^{q-2}|D^2 \ln v|^2  - \frac{q(q-2)}{4}\int_\Omega v^{-q+1}|\nabla v|^{q-4} \left | \nabla |\nabla v|^2 \right |^2 \notag \\
   &\qquad+q(q-2)\int_\Omega v^{-q}|\nabla v|^{q-2}\nabla v \cdot \nabla |\nabla v|^2 -q(q-2) \int_\Omega v^{-q-1}|\nabla v|^{q+2} \notag \\
   &\quad= -q \int_\Omega v^{-q+3}|\nabla v|^{q-2}|D^2 \ln v|^2 - \frac{q(q-2)}{4} \int_\Omega v^{-q+1}|\nabla v|^{q-4} \left | \nabla |\nabla v|^2- \frac{2}{v}|\nabla v|^2 \nabla v \right |^2 \notag \\
   &\quad\leq -q \int_\Omega v^{-q+3}|\nabla v|^{q-2}|D^2 \ln v|^2,
    \end{align}
    while using the inequality $|\Delta v|\leq \sqrt{n}|D^2v|$, we obtain that 
    \begin{align} \label{LW.3}
        -q(q-2)&\int_\Omega uv^{-q+2}|\nabla v|^{q-4} \nabla v \cdot (D^2 v \cdot \nabla v) - q\int_\Omega u v^{-q+2}|\nabla v|^{q-2}\Delta v \notag \\
        &\leq q(q-2) \int_\Omega uv^{-q+2}|\nabla v|^{q-2}|D^2 v|+q\sqrt{n}\int_\Omega uv^{-q+2}|\nabla v|^{q-2}|D^2 v|.
    \end{align}
    Combining \eqref{LW.1}, \eqref{LW.2} and \eqref{LW.3}, we arrive at \eqref{LW-1}.
\end{proof}

Thanks to the above lemma, we derive a following estimate which will be used to derive the boundedness of $\int_\Omega u^pv$ in Lemma \ref{Lpuv}.
\begin{lemma} \label{L2}
For any $p>1$ then there exist $C_1>0$ and $C_2>0$ such that 
    \begin{align*}
        \frac{d}{dt} \int_\Omega \frac{|\nabla v|^{2p}}{v^{2p-1}}+ \int_\Omega \frac{|\nabla v|^{2p}}{v^{2p-1}}+ \frac{2p-1}{(2p+\sqrt{n})^2}\int_\Omega \frac{|\nabla v|^{2p+2}}{v^{2p+1}} &\leq n^{8p+2}p^{4p+2}\int_\Omega u^{p+1}v +C_1 \int_\Omega \frac{|\nabla v|^{2p-1}}{v^{2p-1}} +C_2 ,
    \end{align*}
    for all $t\in (0,T_{\rm max})$.
\end{lemma}
\begin{proof}
    Making use of Lemma 3.3 in \cite{Winkler2022}, we obtain
    \begin{align}  \label{L2.1}
    \frac{d}{dt}\int_\Omega \frac{|\nabla v|^{2p}}{v^{2p-1}}+\int_\Omega \frac{|\nabla v|^{2p}}{v^{2p-1}}  &\leq p \int_{\partial \Omega} \frac{|\nabla v|^{2p-2}}{v^{2p-1}} \frac{\partial |\nabla v|^2}{\partial \nu} - 2p\int_\Omega \frac{|\nabla v|^{2p-2}}{v^{2p-3}}  |D^2 \ln v|^2  +(2p-1)^2\int_\Omega u \frac{|\nabla v|^{2p}}{v^{2p-1}} \notag \\
    &\quad + 2p(2p-2+\sqrt{n}) \int_\Omega   \frac{u|\nabla v|^{2p-2}}{v^{2p-2}} |D^2 v| +c_1\int_\Omega \frac{|\nabla v|^{2p-1}}{v^{2p-1}}
\end{align}
for all $t\in (0,T_{\rm max})$, where $c_1= 2p \left \| \nabla B_2 \right \|_{L^\infty(\Omega \times (0, \infty))} $. Let
    \begin{align*}
        \epsilon= \frac{1}{p+2(2p+\sqrt{n})^2+(2p+\sqrt{n}+1)^2}.
    \end{align*}
     By applying Young's inequality and using \eqref{Lw3-0}, we deduce that   
    \begin{align} \label{L2.2}
       (2p-1)^2 \int_\Omega u\frac{|\nabla v|^{2p}}{v^{2p-1}} &\leq  \epsilon \int_\Omega \frac{|\nabla v|^{2p+2}}{v^{2p+1}} +\frac{(2p-1)^{2p+2}}{\epsilon^p} \int_\Omega u^{p+1}v \notag \\
        &\leq \epsilon (2p+\sqrt{n})^2 \int_\Omega \frac{|\nabla v|^{2p-2}}{v^{2p-3}}  |D^2 \ln v|^2 +\frac{(2p-1)^{2p+2}}{\epsilon^p} \int_\Omega u^{p+1}v. 
    \end{align}
     In view of \eqref{Lw3}, we can find $c_2=c_2(\epsilon)>0$ such that 
    \begin{align} \label{L2.3}
        p \int_{\partial \Omega} \frac{|\nabla v|^{2p-2}}{v^{2p-1}} \frac{\partial |\nabla v|^2}{\partial \nu} \leq  \epsilon p \int_\Omega \frac{|\nabla v|^{2p-2}}{v^{2p-3}}  |D^2 \ln v|^2 +c_2 \int_\Omega v.
    \end{align}
    Applying Young's inequality and employing \eqref{Lw3-0} and \eqref{Lw3-1}, we infer that 
    \begin{align} \label{L2.4}
         2p(2p-2+\sqrt{n}) \int_\Omega   \frac{u|\nabla v|^{2p-2}}{v^{2p-2}} |D^2 v| &\leq \epsilon  \int_\Omega \frac{|\nabla v|^{2p-2}}{v^{2p-3}}  |D^2  v|^2 + \frac{p^2(2p-2+\sqrt{n})^2}{\epsilon} \int_\Omega u^2 \frac{|\nabla v|^{2p-2}}{v^{2p-3}} \notag \\
         &\leq \epsilon  \int_\Omega \frac{|\nabla v|^{2p-2}}{v^{2p-3}}  |D^2  v|^2 + \epsilon \int_\Omega \frac{|\nabla v|^{2p+2}}{v^{2p+1}}+ \frac{ \left (p(2p-2+\sqrt{n}) \right)^{p+1}}{\epsilon^p} \int_\Omega u^{p+1}v \notag \\
         &\leq \epsilon \left ( (2p+\sqrt{n}+1)^2+ (2p+\sqrt{n})^2 \right )\int_\Omega \frac{|\nabla v|^{2p-2}}{v^{2p-3}}  |D^2 \ln v|^2  \notag \\
         &\quad+ \frac{ \left (p(2p-2+\sqrt{n}) \right)^{p+1}}{\epsilon^p} \int_\Omega u^{p+1}v.
    \end{align}
    Now, collecting from \eqref{L2.1} to \eqref{L2.4} and applying \eqref{Lw3-0} and \eqref{L1est-2}, we infer that 
    \begin{align*}
         \frac{d}{dt}\int_\Omega \frac{|\nabla v|^{2p}}{v^{2p-1}}+\int_\Omega \frac{|\nabla v|^{2p}}{v^{2p-1}} &\leq -\left (2p - \epsilon \left( p +2(2p+\sqrt{n})^2 +(2p+\sqrt{n}+1)^2\right) \right )\int_\Omega \frac{|\nabla v|^{2p-2}}{v^{2p-3}}  |D^2 \ln v|^2   \notag \\
         &\quad+ \frac{ \left (p(2p-2+\sqrt{n}) \right)^{p+1} +(2p-1)^{2p+2}}{\epsilon^p} \int_\Omega u^{p+1}v + c_1\int_\Omega \frac{|\nabla v|^{2p-1}}{v^{2p-1}}+c_2\int_\Omega v \notag \\
         &= -(2p-1) \int_\Omega \frac{|\nabla v|^{2p-2}}{v^{2p-3}}  |D^2 \ln v|^2+ \frac{ \left (p(2p-2+\sqrt{n}) \right)^{p+1} +(2p-1)^{2p+2}}{\epsilon^p} \int_\Omega u^{p+1}v \notag \\
         &\quad + c_1\int_\Omega \frac{|\nabla v|^{2p-1}}{v^{2p-1}}+c_2\int_\Omega v \notag \\
         &\leq -\frac{2p-1}{(2p+\sqrt{n})^2} \int_\Omega \frac{|\nabla v|^{2p+2}}{v^{2p+1}}  + c_1\int_\Omega \frac{|\nabla v|^{2p-1}}{v^{2p-1}}+c_2\int_\Omega v \notag \\ 
         &\quad+ \left [\left (p(2p-2+\sqrt{n}) \right)^{p+1} +(2p-1)^{2p+2} \right ]\cdot \left ( p+ 2(2p+\sqrt{n})^2+(2p+\sqrt{n}+1)^2 \right )^p\int_\Omega u^{p+1}v \notag \\
         &\leq -\frac{2p-1}{(2p+\sqrt{n})^2} \int_\Omega \frac{|\nabla v|^{2p+2}}{v^{2p+1}}  + c_1\int_\Omega \frac{|\nabla v|^{2p-1}}{v^{2p-1}}+ n^{8p+2}p^{4p+2}\int_\Omega u^{p+1}v+c_3
    \end{align*}
 for all $t\in (0,T_{\rm max})$,   where $c_3=c_2 \sup_{t\in (0,T_{\rm max})}\int_\Omega v(\cdot,t)$ and the last inequality holds due to
 \begin{align*}
     \left [\left (p(2p-2+\sqrt{n}) \right)^{p+1} +(2p-1)^{2p+2} \right ]\cdot \left ( p+ 2(2p+\sqrt{n})^2+(2p+\sqrt{n}+1)^2 \right )^p \leq n^{8p+2}p^{4p+2}.
 \end{align*}
 The proof is now complete.
\end{proof}

Before establishing the key estimate, namely the uniform-in-time boundedness of $\int_\Omega u^pv$, we will need the following identity, which explains the role of the first term of \eqref{F}.
\begin{lemma}\label{L1}
    For any $p>1$, the following holds
 \begin{align} \label{L1-1}
        \frac{d}{dt}\int_\Omega u^p v&=-p(p-1)\int_\Omega u^{p-2}v|\nabla u|^2+p(\chi(p-1)-2)\int_\Omega u^{p-1}\nabla u \cdot \nabla v +\chi p \int_\Omega u^p \frac{|\nabla v|^2}{v} \notag \\
       &\quad +p\int_\Omega u^{p-1}v B_1+ \int_\Omega u^p B_2+(rp-1) \int_\Omega u^pv - (\mu p-1) \int_\Omega u^{p+1}v \qquad \text{for all }t\in (0,T_{\rm max}).
   \end{align}
\end{lemma}

\begin{proof}
    Direct calculations shows that 
    \begin{align} \label{L1.1}
        \frac{d}{dt}\int_\Omega u^p v &= p \int_\Omega u^{p-1}v u_t+ \int_\Omega u^pv_t \notag \\
        &:=I+J.
    \end{align}
   Using the first equation of \eqref{1}  and integrating by parts, we obtain 
   \begin{align} \label{L1.2}
       I&=p \int_\Omega u^{p-1}v \left ( \Delta u - \chi \nabla \cdot \left ( u \frac{\nabla v}{v} \right ) -uv+B_1+ru - \mu u^2  \right ) \notag \\
       &= -p(p-1)\int_\Omega u^{p-2}v|\nabla u|^2+p(\chi(p-1)-1)\int_\Omega u^{p-1}\nabla u \cdot \nabla v +\chi p \int_\Omega u^p \frac{|\nabla v|^2}{v} \notag \\
       &\quad -p\int_\Omega u^p v^2+p\int_\Omega u^{p-1}v B_1+rp \int_\Omega u^pv - \mu p \int_\Omega u^{p+1}v \qquad \text{for all }t\in (0,T_{\rm max}).
   \end{align}
   Now, using the second equation of \eqref{1}  and integrating by parts, we have
   \begin{align} \label{L1.3}
       J&= \int_\Omega u^p \left ( \Delta v -v+uv+B_2 \right ) \notag \\
       &=-p \int_\Omega u^{p-1}\nabla u \cdot \nabla v - \int_\Omega u^pv+ \int_\Omega u^{p+1}v + \int_\Omega u^p B_2 \qquad \text{for all }t\in (0,T_{\rm max}).
   \end{align}
   Combining \eqref{L1.1}, \eqref{L1.2} and \eqref{L1.3}, we arrive at \eqref{L1-1}. The proof is now complete.
\end{proof}

We are now ready to establish the most important estimate, which allows us to derive $L^p$ boundedness for $u$ in Lemma \ref{Lp}.

\begin{lemma} \label{Lpuv}
    Let $p>1$ and suppose that $\mu>\mu_0(p)$ where $\mu_0(p)$ is defined in \eqref{mu}. If $T_{\rm max}<\infty$, then there exists $C=C(p,T_{\rm max})>0$ such that
    \begin{align} \label{Lpuv-1}
        \int_\Omega u^p(\cdot,t)v(\cdot,t) \leq C \qquad \text{for all } t\in (0,T_{\rm max}),
    \end{align}
    and
    \begin{align} \label{Lpuv-2}
        \int_\Omega \frac{|\nabla v(\cdot,t)|^{2p}}{v^{2p-1}(\cdot,t)} \leq C \qquad \text{for all } t\in (0,T_{\rm max}).
    \end{align}
    Moreover, if condition \eqref{B2} holds, then \eqref{Lpuv-1} and \eqref{Lpuv-2} also hold for $T_{\rm max}\in(0,\infty]$, with $C$ independent of $T_{\rm max}$.
\end{lemma}

\begin{proof}
Employing Lemma \ref{L1} and applying Young's inequality, we obtain 
 \begin{align} \label{Lpuv.1}
        \frac{d}{dt}\int_\Omega u^p v&=-p(p-1)\int_\Omega u^{p-2}v|\nabla u|^2+p(\chi(p-1)-2)\int_\Omega u^{p-1}\nabla u \cdot \nabla v +\chi p \int_\Omega u^p \frac{|\nabla v|^2}{v} \notag \\
       &\quad +p\int_\Omega u^{p-1}v B_1+ \int_\Omega u^p B_2+(rp-1) \int_\Omega u^pv - (\mu p-1) \int_\Omega u^{p+1}v  \notag \\
       &\leq c_1 \int_\Omega u^p \frac{|\nabla v|^2}{v}+ c_2 \int_\Omega u^{p-1}v+c_3\int_\Omega u^p \notag \\
       &\quad+(rp-1) \int_\Omega u^pv - (\mu p-1) \int_\Omega u^{p+1}v  
       \qquad \text{for all }t\in (0,T_{\rm max}),
   \end{align}
   where $c_1= \chi p + \frac{p(\chi(p-1)-2)^2}{4(p-1)}$, $c_2= p \left \|B_1 \right \|_{L^\infty(\Omega \times (0, \infty))}$ and $c_3=  \left \|B_2 \right \|_{L^\infty(\Omega \times (0, \infty))}$. Two further application Young's inequality yields 
    \begin{align} \label{Lpuv.2}
      c_1 \int_\Omega u^p \frac{|\nabla v|^2}{v} \leq  \frac{2p-1}{(2p+\sqrt{n})^2}\int_\Omega \frac{|\nabla v|^{2p+2}}{v^{2p+1}}+c_4 \int_\Omega u^{p+1}v, 
  \end{align}
  where $c_4= c_1^{\frac{p+1}{p}} \cdot \left ( \frac{(2p+\sqrt{n})^2}{2p-1} \right )^\frac{1}{p}$ and 
  \begin{align}\label{Lpuv.3}
      c_2 \int_\Omega u^{p-1}v+ rp \int_\Omega u^pv &\leq \int_\Omega u^{p+1}v + c_5 \int_\Omega v\notag \\
      &\leq \int_\Omega u^{p+1}v  +c_6,
  \end{align}
  where $c_5>0$ and $c_6=c_5\sup_{t\in (0,T_{\rm max})} \int_\Omega v(\cdot,t) $ is finite due to \eqref{L1est-2}. In view of Lemma \ref{L2}, we can find $c_7>0$ and $c_8>0$ such that
   \begin{align}\label{Lpuv.4}
        \frac{d}{dt} \int_\Omega \frac{|\nabla v|^{2p}}{v^{2p-1}}+ \int_\Omega \frac{|\nabla v|^{2p}}{v^{2p-1}}+ \frac{2p-1}{(2p+\sqrt{n})^2}\int_\Omega \frac{|\nabla v|^{2p+2}}{v^{2p+1}} &\leq n^{8p+2}p^{4p+2}\int_\Omega u^{p+1}v +c_7 \int_\Omega \frac{|\nabla v|^{2p-1}}{v^{2p-1}} +c_8 ,
    \end{align}
    for all $t\in (0,T_{\rm max})$. We apply Lemma \ref{low} to deduce that
    \begin{align*}
       \min_{x\in \bar{\Omega}} v(x,t) \geq \delta \qquad \text{for all }t\in (0,T_{\rm max}),
    \end{align*}
    where $\delta =e^{-T_{\rm max}} \inf_{x\in \Omega }v_0(x)>0$ if $T_{\rm max}< \infty$ and $\delta>0$ is independent of $T_{\rm max}$ if \eqref{B2} holds. Using this, together with Young's inequality, we obtain 
    \begin{align}\label{Lpuv.6}
        c_3\int_\Omega u^p &\leq \int_\Omega u^{p+1}v + c_2^{p+1}\int_\Omega v^{-p} \notag \\
        &\leq \int_\Omega u^{p+1}v +c_9,
    \end{align}
    where $c_9=c_3^{p+1}\delta^{-p}|\Omega|$, and
    \begin{align}\label{Lpuv.7}
        c_7 \int_\Omega \frac{|\nabla v|^{2p-1}}{v^{2p-1}} &\leq \frac{1}{2} \int_\Omega  \frac{|\nabla v|^{2p}}{v^{2p-1}} + 2^{2p-1}c_7^{2p}\int_\Omega v^{1-2p} \notag \\
        &\leq \frac{1}{2} \int_\Omega  \frac{|\nabla v|^{2p}}{v^{2p-1}}+c_{10},
    \end{align}
    where $c_{10} = 2^{2p-1}c_7^{2p} \delta^{1-2p}|\Omega|$. Setting 
    \begin{align*}
        F(t)= \int_\Omega u^pv + \int_\Omega \frac{|\nabla v|^{2p}}{v^{2p-1}} \qquad \text{for all }t\in (0,T_{\rm max}),
    \end{align*}
  and collecting from \eqref{Lpuv.1} to \eqref{Lpuv.7} and noting from \eqref{mu} that 
    \begin{align*}
      \mu_0(p)= \frac{3+c_4}{p} +n^{8p+2}p^{4p+1},
  \end{align*}
  we arrive at 
  \begin{align*}
      F'(t)+\frac{1}{2}F(t) &\leq -\left (\mu p  -3-c_4-n^{8p+2}p^{4p+2} \right )\int_\Omega u^{p+1}v+c_{11} \notag \\
      &= -p(\mu-\mu_0(p))\int_\Omega u^{p+1}v+c_{11} \leq c_{11} \qquad \text{for all }t\in (0,T_{\rm max}),
  \end{align*}
  where $c_{11}= c_6+c_9+c_{10}$ and the last inequality holds because $\mu>\mu_0(p)$. This, together with Gronwall's inequality implies \eqref{Lpuv-1} and \eqref{Lpuv-2}, which completes the proof.
\end{proof}
\section{$L^p$ boundedness and proof of the main result} \label{S4}

In this section, we shall establish the $L^p $ boundedness for $u$ and provide a proof for our main result.  As a consequence of Lemma \ref{Lpuv}, we can derive $L^p$ boundedness for $u$ and $W^{1, \infty}$ boundedness for $v$ in the following lemma. 
\begin{lemma} \label{Lp}
    Assume that $\mu> \mu_*$ where $\mu_*$ is defined in \eqref{mu*}. If $T_{\rm max}< \infty $ then there exist $p> n$ and $C=C(p,T_{\rm max})>0$ such that 
    \begin{align} \label{Lp-1}
        \int_\Omega u^{p}(\cdot,t) \leq C \qquad \text{for all }t\in (0,T_{\rm max}),
    \end{align}
    and 
    \begin{align} \label{Lp-2}
        \left \| v (\cdot,t) \right \|_{W^{1, \infty}(\Omega)}  \leq C \qquad \text{for all }t\in (0,T_{\rm max}).
    \end{align}
    Furthermore, if the condition \eqref{B2} holds then \eqref{Lp-1} and \eqref{Lp-2} also hold for $T_{\rm max} \in (0, \infty]$, with $C$ independent of $T_{\rm max}$.
\end{lemma}
\begin{proof}
   By the definition of $\mu^*$, the continuity of $\mu_0(\cdot)$ and
the assumption $\mu > \mu^* = \mu_0(n)$, we can find some 
$p > n$ such that $\mu > \mu_0(p)$. Therefore, if either $T_{\rm max} < \infty$ 
or condition \eqref{B2} holds, we may apply Lemma \ref{Lpuv} to obtain
\begin{align}
    \label{Lp.1}
    \int_\Omega u^p(\cdot,t)\, v(\cdot,t) \le c_1 
    \qquad \text{for all } t \in (0,T_{\rm max}),
\end{align}
and
\begin{align}
    \label{Lp.2}
    \left \| \nabla v^{\frac{1}{2p}}(\cdot,t) \right \|_{L^{2p}(\Omega)} \le c_1 
    \qquad \text{for all } t \in (0,T_{\rm max}),
\end{align}
where $c_1 > 0$ depends on $T_{\rm max}$ when $T_{\rm max} < \infty$, 
and is independent of $T_{\rm max}$ when \eqref{B2} is fulfilled. In view of Lemma \ref{low}, we infer that 
    \begin{align*}
       \min_{x\in \bar{\Omega}} v(x,t) \geq \delta \qquad \text{for all }t\in (0,T_{\rm max}),
    \end{align*}
    where $\delta =e^{-T_{\rm max}} \inf_{x\in \Omega }v_0(x)>0$ if $T_{\rm max}< \infty$ and $\delta>0$ is independent of $T_{\rm max}$ if \eqref{B2} holds. This, together with \eqref{Lp.1} implies that 
    \begin{align} \label{Lp.3}
          \int_\Omega u^p(\cdot,t) \le \frac{1}{\delta} \int_\Omega u^p(\cdot,t)v(\cdot,t) \leq \frac{c_1}{\delta}
    \qquad \text{for all } t \in (0,T_{\rm max}).
    \end{align}
    Because $2p> n$, we invoke the Sobolev's embedding $W^{1,2p}(\Omega) \hookrightarrow  L^\infty(\Omega)$, together with \eqref{L1est-2} to deduce that 
    \begin{align*}
        \left \| v^{\frac{1}{2p}}(\cdot,t) \right \|_{L^\infty(\Omega)} \leq \left \| \nabla v^{\frac{1}{2p}}(\cdot,t) \right \|_{L^{2p}(\Omega)} + \left \| v^\frac{1}{2p}(\cdot,t) \right \|_{L^{2p}(\Omega)} \leq c_2 \qquad \text{for all } t \in (0,T_{\rm max}),
    \end{align*}
    where $c_2=c_1+ \sup_{t\in (0,T_{\rm max})}\left \| v(\cdot,t) \right \|^{\frac{1}{2p}}_{L^1(\Omega)}$, which further entails that 
    \begin{align} \label{Lp.4}
        \left \| v(\cdot,t) \right \|_{L^\infty(\Omega)} \leq c_2^{2p}\qquad \text{for all } t \in (0,T_{\rm max}).
    \end{align}
    From \eqref{Lp.3}, \eqref{Lp.4} and \eqref{B}, we find that 
    \begin{align*}
       \left \| u(\cdot,t)v(\cdot,t)+B_2(\cdot,t) \right \|_{L^p(\Omega)}  \leq \left \| v(\cdot,t) \right \|_{L^\infty(\Omega)} \left \| u(\cdot,t) \right \|_{L^p(\Omega)} +  \left \| B_2 \right \|_{L^\infty(\Omega\times (0, \infty))}\leq c_3\qquad \text{for all } t \in (0,T_{\rm max}).\,
    \end{align*}
    where $c_3=c_2^{2p}c_1^{\frac{1}{p}}\delta^{-\frac{1}{p}}+\left \| B_2 \right \|_{L^\infty(\Omega\times (0, \infty))}$. This enables us to apply Lemma \ref{C52.Para-Reg} with $f=uv+B_2$ to deduce that 
    \begin{align*}
     \left \| v (\cdot,t) \right \|_{W^{1, \infty}(\Omega)}  \leq C \qquad \text{for all }t\in (0,T_{\rm max}),
    \end{align*}
    which, together with \eqref{Lp.3} completes the proof.
\end{proof}
We are now in position to prove our main result.
\begin{proof}[Proof of Theorem \ref{thm1}]
   We argue by contradiction. Suppose that the solution is not global, i.e. \(T_{\max} < \infty\).  From Lemma \ref{low}[\eqref{low-1}], we can find $c_1=c_1(T_{\rm max})>0$ such that 
     \begin{align} \label{prf.3}
        \left \| \frac{1}{v (\cdot,t) }\right \|_{L^{\infty}(\Omega)}  \leq c_1 \qquad \text{for all }t\in (0,T_{\rm max}).
    \end{align}
    Then we apply Lemma \ref{Lp} to deduce that there exists some $p_0>n$ such that  
    \begin{align} \label{prf.1}
        \int_\Omega u^{p_0}(\cdot,t) \leq c_2 \qquad \text{for all }t\in (0,T_{\rm max}),
    \end{align}
    and 
    \begin{align} \label{prf.2}
        \left \| v (\cdot,t) \right \|_{W^{1, \infty}(\Omega)}  \leq c_2 \qquad \text{for all }t\in (0,T_{\rm max}),
    \end{align}
    where $c_2=c_2(T_{\rm max})>0$. Now,  multiplying the first equation by $u^{p-1}$ with $p>1$, integrating by parts, applying Young's inequality and using \eqref{prf.3} and \eqref{prf.2}, we infer that 
    \begin{align*}
       \frac{1}{p} \frac{d}{dt} \int_\Omega u^p &=-(p-1)\int_\Omega u^{p-2}|\nabla u|^2+ \chi(p-1)\int_\Omega u^{p-1} \nabla u \cdot \frac{\nabla v}{v} - \int_\Omega u^pv +\int_\Omega B_1 u^{p-1} +r\int_\Omega u^p - \mu \int_\Omega u^{p+1} \notag \\
       &\leq -\frac{p-1}{2}\int_\Omega u^{p-2}|\nabla u|^2+ \frac{\chi^2(p-1)}{2} \int_\Omega u^p \frac{|\nabla v|^2}{v^2} + \left \| B_1 \right \|_{L^\infty \left ( \Omega \times (0, \infty) \right )} \int_\Omega u^{p-1} +r\int_\Omega u^p \notag \\
       &\leq -\frac{p-1}{2}\int_\Omega u^{p-2}|\nabla u|^2+ \left (r+ \frac{\chi^2(p-1)c_1^2c_2^2}{2} \right ) \int_\Omega u^p  + \left \| B_1 \right \|_{L^\infty \left ( \Omega \times (0, \infty) \right )} \int_\Omega u^{p-1}
    \end{align*}
    for all $t\in (0,T_{\rm max})$. Now, by applying standard Moser's iteration method  (see \cite{Alikakos1, Alikakos2, Winkler-2011}), we deduce that  
    \begin{align} \label{prf.4}
        \left \| u (\cdot,t) \right \|_{L^{ \infty}(\Omega)}  \leq c_3 \qquad \text{for all }t\in (0,T_{\rm max}),
    \end{align}
    for some $c_3=c_3(T_{\rm max})>0$. Combining \eqref{prf.3}, \eqref{prf.2} and \eqref{prf.4}, we arrive at 
    \begin{align*}
         \left \| u (\cdot,t) \right \|_{L^{ \infty}(\Omega)}+ \left \| v (\cdot,t) \right \|_{W^{1, \infty}(\Omega)} + \left \| \frac{1}{v (\cdot,t) }\right \|_{L^{\infty}(\Omega)} \leq c_4\qquad \text{for all }t\in (0,T_{\rm max}),
    \end{align*}
    where $c_4=c_1+c_2+c_3$. This, however, contradicts the extensibility criterion stated in \eqref{local-1}. Hence, the assumption \(T_{\max} < \infty\) is false, and therefore \(T_{\max} = \infty\). Now, suppose that \eqref{B2} holds, then by invoking \eqref{low-2} and Lemma \ref{Lp}, we infer that the estimates \eqref{prf.3}, \eqref{prf.1} and \eqref{prf.2} hold with positive constants $c_1,c_2$ and $c_3$ independent of $T_{\rm max}$. Hence, we deduce that \eqref{prf.4} also holds with $c_4>0$ independent of $T_{\rm max}$. Therefore, we obtain 
    \begin{align*}
        \sup_{t>0}  \left \{ \left \| u (\cdot,t) \right \|_{L^{ \infty}(\Omega)}+ \left \| v (\cdot,t) \right \|_{W^{1, \infty}(\Omega)} \right \} < \infty,
    \end{align*}
    which finishes the proof.
\end{proof}

\section*{Declarations}
\paragraph{Conflict of Interest} The author declares that they have no conflict of interest.
\paragraph{Acknowledgments} Minh Le was supported by the Hangzhou Postdoctoral Research Grant.

     \paragraph{Data Availability}
 Data sharing not applicable to this article as no datasets were generated or analyzed during
the current study.

\end{document}